\documentclass{IEEEtran4PSCC}
\ifCLASSINFOpdf
  \usepackage[pdftex]{graphicx}
\else
  \usepackage[dvips]{graphicx}
\fi
\usepackage[caption=false]{subfig}
\usepackage{float}
\usepackage[cmex10]{amsmath}
\usepackage{amsfonts}
\usepackage{amssymb}
\usepackage[usenames,dvipsnames]{xcolor}
\usepackage{amsthm}
\newtheorem{theorem}{Theorem}[section]

\newtheorem{assumption}{Assumption}

\DeclareMathOperator*{\cost}{cost}
\DeclareMathOperator*{\st}{s.t.:}
\usepackage{booktabs}
\usepackage{algorithm}
\usepackage{algorithmic}
\usepackage{color}
\usepackage[urlcolor=true]{hyperref}
\usepackage{cleveref}
\usepackage{lipsum}

\usepackage{pgfplots}
\pgfplotsset{compat=newest}
\usepackage{tikz}
\usetikzlibrary{calc}
\usetikzlibrary{arrows, arrows.meta}
\newcommand{\cE}{{\cal E}}

\newcommand{\cP}{{\cal P}}
\newcommand{\cN}{{\cal N}}
\newcommand{\cG}{{\cal G}}

\usepackage{amssymb}
\usepackage{amsmath}
\makeatletter
\let\old@ps@headings\ps@headings
\let\old@ps@IEEEtitlepagestyle\ps@IEEEtitlepagestyle
\def\psccfooter#1{%
  \def\ps@headings{%
    \old@ps@headings%
    \def\@oddfoot{\strut\hfill#1\hfill\strut}%
    \def\@evenfoot{\strut\hfill#1\hfill\strut}%
  }%
  \def\ps@IEEEtitlepagestyle{%
    \old@ps@IEEEtitlepagestyle%
    \def\@oddfoot{\strut\hfill#1\hfill\strut}%
    \def\@evenfoot{\strut\hfill#1\hfill\strut}%
  }%
  \ps@headings%
}
\makeatother

\begin{document}
%
\title{Importance Sampling Approach to Chance-Constrained DC Optimal Power Flow}

\author{
\IEEEauthorblockN{Aleksander Lukashevich, Vyacheslav Gorchakov, Petr Vorobev}
\IEEEauthorblockA{Center for Energy Systems and Technology\\
Skolkovo Institute of Science and Technology\\
Skokovo Innovation Center, Moscow Region, Russia\\
\{aleksandr.lukashevich, v.gorchakov, p.vorobev\}@skoltech.ru}
\and
\IEEEauthorblockN{Deepjyoti Deka, Yury Maximov}
\IEEEauthorblockA{Theoretical Division, T-5 group, \\
Los Alamos National Laboratory\\
Los Alamos, New Mexico, USA\\
\{deepjyoti, yury\}@lanl.gov}
}



\maketitle

\begin{abstract}
Despite significant economic and ecological effects, a higher level of renewable energy generation leads to increased uncertainty and variability in power injections, thus compromising grid reliability. In order to improve power grid security, we investigate a joint chance-constrained (CC) direct current (DC) optimal power flow (OPF) problem. The problem aims to find economically optimal power generation while guaranteeing that all power generation, line flows, and voltages simultaneously remain within their bounds with a pre-defined probability. Unfortunately, the problem is computationally intractable even if the distribution of renewables fluctuations is specified. Moreover, existing approximate solutions to the joint CC OPF problem are overly conservative, and therefore have less value for the operational practice. This paper proposes an importance sampling approach to the CC DC OPF problem, which yields better complexity and accuracy than current state-of-the-art methods. The algorithm efficiently reduces the number of scenarios by generating and using only the most important of them, thus enabling real-time solutions for test cases with up to several hundred buses. 
\end{abstract}

\begin{IEEEkeywords}
robust optimization, chance-constrained optimization, optimal power flow
\end{IEEEkeywords}


\section{Introduction}

In 2020 electricity produced approximately 25\% of greenhouse gas emissions in the USA, and integration of a higher volume of renewable energy generation is seen as the primary tool to reduce the emission level. In turn, a higher amount of renewable generation increases the power grid uncertainty, compromises its security, and challenges classical power grid operation and planning policies. 

The optimal power flow (OPF) problem, which determines the economically optimal operating level of power generation under given power balance equations and security constraints, is one of the most fundamental problems in grid operation and planning. Several extensions are proposed for the optimal power flow problem for addressing the uncertainty of power generation and consumption. Robust and chance-constrained power flow formulations are among the most popular ones. The robust OPF problem assumes bounded uncertainty and requires a solution to be feasible against any possible uncertainty realization within a given uncertainty set~\cite{ben2002robust,ding2016adjustable,sousa2010robust}. 

A more flexible chance-constrained approach requires security constraints to be satisfied a high probability while assuming the distribution of renewables is known in whole or in part \cite{lubin2015robust, roald2017chance,bienstock2014chance,pena2020dc}. This paper considers a {\it joint chance-constrained} optimal power flow problem, where the joint probability of at least one failure of the security constraints (line load limits, voltage stability bounds) is bounded from above by a confidence threshold (JCC-OPF). 

In contrast to a {\it single chance-constrained} formulation (SCC-OPF), which imposes individual failure probability thresholds for each of the constrains, the joint chance constraint is computationally hard even for (linear) direct current (DC) power balance equations, linear security limits and Gaussian uncertainty~\cite{cousins2014cubic,khachiyan1993complexity}. Several tractable convex approximations have been proposed \cite{nemirovski2007convex, nemirovski2006scenario,nemirovski2003tractable,trofino1999bi} for joint chance-constrained optimization to overcome the computational hardness of the problem; however, they often lead to conservative solutions inapplicable for operational practice. 

Other approaches such as scenario and sample average approximations \cite{calafiore2006scenario,garatti2019risk, vrakopoulou2013probabilistic,nemirovski2006scenario} consist of substituting the stochastic part with a set of deterministic inequalities based on the uncertainty realization. This approach can be distributionally robust and allow exploiting uncertainties beyond the Gaussian ones. A combination of an analytical approximation and sampling~\cite{hou2020chance} can further improve the accuracy of the solution. However it may require a large number of samples. Scenario curation/modification heuristics have been designed to improve the sample complexity of JCC-OPF \cite{mezghani2020stochastic}, although with formal analysis of the methodology. In work outside power grids, statistical learning has been used to approximate uncertain convex programs~\cite{vapnik1999overview,maximov2016tight,campi2020scenario}.

Nevertheless, the scenario approximation approach remains the most accurate among algorithms for solving the joint chance-constrained DC optimal power flow, although its complexity is often unacceptable for large-scale power grids~\cite{sakhavand2020new}. To this end, the paper suggests using importance sampling to reduce the complexity and improve the accuracy of the scenario approximation to chance-constrained optimal power flow. Importance sampling allows generating more informative samples and results in an optimization problem with fewer constraints. 

The contribution of the paper is as follows. 
\begin{enumerate}
  \item we propose a novel computationally efficient 
  approach to the joint chance-constrained DC OPF problem. The algorithm exploits physics-informed importance sampling to refine the classical scenario approximation~\cite{calafiore2006scenario};
  \item we prove the algorithm to converge to a sub-optimal solution with a guaranteed accuracy under mild technical assumptions; 
  \item we demonstrate the proposed algorithm's high time efficiency and accuracy over multiple real and synthetic test cases. 
\end{enumerate}

The paper is organized as follows.
In Section~\ref{sec:not} we present the joint chance-constrained optimal power flow problem and introduce notation used in the paper. We outline the algorithm and provide its theoretical analysis in Section~\ref{sec:algo}. Empirical study and comparison to state-of-the-art methods are given in Section~\ref{sec:emp}. In Section~\ref{sec:conclusion} we conclude with a brief summary and discussion on possible applications of our results.

\section{Background and Problem Setup}\label{sec:not}

\subsection{Notation}
The direct current (DC) power flow model remains extremely popular yet simple for the analysis because of a linear relation between powers and phase angles that is typical high-voltage power grids. 

In the following, we consider a power grid given by a graph $\cG = (V, E)$ with a set of nodes/buses $V$ and edges/lines $E$. Assume $n$ is a number of buses, and $m$ is a number of lines. Let $p$ be a vector of power injections $p = (p_F, p_R, p_S)^\top$, where $p_F$ corresponds to buses with deterministic/fixed (F) power injections, $p_R$ to buses with random (R) injections, and $p_S$ is the injection at the slack bus (S). The power system is balanced, i.e. the sum of all power injections equals to zero, $p_S + \sum_{r\in R} p_r + \sum_{f\in F} p_f = 0$. Let $\theta$ be a vector of phase angles. Without a loss of generalization, We assume that the phase angle on the reference slack bus $\theta_S = 0$. Let $B \in \mathbb{R}^{N \times N}$ be an admittance matrix of the system, $p = B\theta$. The components $B_{ij}$ are such that $B_{ij} \neq 0$ if there is a line between nodes $i$ and $j$, the diagonal elements $B_{ii} = - \sum_{j\neq i} B_{ij}$, i.e. $B$ is a Laplacian matrix. Let $B^{\dagger}$ be the pseudo-inverse of $B$, i.e.~$\theta = B^{\dagger} p$. 

The DC power flow equations, generation and stability constraints are
\begin{align}
    & p = B \theta \label{eq:DC-PF-a}\\
    & p^{\min}_i \leq p_i \leq p^{\max}_i, i \in V \label{eq:DC-PF-b}\\
    & |\theta_i - \theta_j| \leq \bar{\theta}_{ij}, \; (i,j)\in E \label{eq:DC-PF-c}
\end{align}
Let $A \in \{-1,0,1 \}^{m \times n}$ be the incidence matrix of grid $G$, i.e. if nodes $i$ and $j$ are connected by edge $k$ then $A_{ki} = +1$, $A_{jk} = -1$ and all other elements are equal to zero. Then the phase angle constraints in \eqref{eq:DC-PF-c} can be represented as  $AB^{\dagger}p \le \bar\theta$, $-AB^{\dagger}p \le \bar\theta$. 

Let $C \in \mathbb{R}^{n\times n}$ be a symmetric matrix such that for all random nodes $r$, fixed nodes $f$ and the slack node $S$, $C_{rr} = 1$, $C_{ff} = 1$, $C_{Sr} = - 1$, $C_{Sf} = - 1$, while all other elements are equal to zero. The Dc OPF constraints are given by the following system of inequalities 
\[
\underbrace{(AB^\dagger C, - A B^\dagger C, C, -C)}_{W}\!\!{}^\top p \le \underbrace{(\bar\theta, \bar\theta, p_{\max}, p_{\min})}_b\!\!{}^\top. 
\]
Let $J$ be a number of constraints, $J = 2m + 2n$, then $W\in\mathbb{R}^{J\times n}$ and $b\in\mathbb{R}^{n\times 1}$. We refer a feasibility polytope as a set~$\cP$, 
$\cP = \bigl\{p:\, Wp \le b\bigr\} = \bigl\{p:\; \bigcap_{i=1}^J\omega_i^\top p \le b_i\bigr\}$.

Finally, we assume that fluctuations of power injections $p$ are Gaussian, $p = x +\xi$, where $\xi$ is a Gaussian uncertainty, $\xi\sim\cN(0, \Sigma)$ and $x$ is the deterministic part of power injections. 

The paper notation is summarized in Table~\ref{tab:notation}. We use lower indices for coordinates of vectors and matrices, lower-case letters for probability density functions (PDFs), and upper-case letters for cumulative distribution functions (CDFs). When it does not lead to confusion, we use $\mathbb{P}$ and $\mathbb{E}$ to denote probability and expectation without explicitly mentioning a distribution. 

\begin{table}[t]
  \centering
  \begin{tabular}{l|l|l|l}
    $E$ & set of lines & $\nu(p)$ & nominal distribution PDF\\
    $V$ & set of buses & ${\cal V}(p)$ & nominal distribution CDF \\
    $B$ & admittance matrix& $\nu(p,x)$ & parametric distribution CDF\\
    $m$ & number of lines & ${\cal V}(p,x)$& parametric distribution CDF\\
    $n$ & number of buses & $p^{\max}$ & generation upper limits\\
    $p$ & power injections & $p^{\min}$& generation lower limits\\
    $\theta$ & voltage phases & $\mathbb{P}(\cdot)$, $\mathbb{E}(\cdot)$ & probability, expectation\\
    $I_n$ & $n\times n$ identity matrix & $\xi$ & vector of stochastic \\
    $J$ & number of constraints & & power fluctuations\\
    $\bar\theta_{ij}$ & voltage angle limits & ${\cal N}(\mu, \Sigma)$ & Gaussian distribution\\
    $\xi$ & power injection unce-& $\cP$ & feasibility polytope\\ 
     &rtainty, $p\sim\cN(0, \Sigma)$
     & &$\cP = \{p:\bigcap_{i\le J} \omega_i^\top p\le b_i\}$\\
    $p$ & power injections & 
    $x$ & deterministic part of \\
    & $p = x + \xi$ & 
    & power injections, $x = \mathbb{E}_\xi p$
  \end{tabular}
  \caption{Paper notation.}
  \label{tab:notation}
\end{table}

\subsection{Problem Setup}
The joint chance-constrained optimal power flow problem~is:
\begin{align}\label{eq:JCC-OPF}
  \min_x & \;\mathbb{E}_{\xi\sim \cN(0, \Sigma)} \cost(x,\xi)\\
   \st\; & \mathbb{P}_{\xi\sim \cN(0, \Sigma)} (x+\xi \in \cP) \ge 1 - \eta, 0 < \eta \le 1/2,\nonumber
\end{align}
where $\eta$ is a preset confidence parameter, and $\cost$ is a cost function convex in $x$ for any realization of $\xi$. In other words, we assume that power flow balance equations (Eqs.~\eqref{eq:DC-PF-a}) are satisfied almost surely, and the probability that at least one of the security constraints (Eqs.~\eqref{eq:DC-PF-b} and~\eqref{eq:DC-PF-c}) fails is at most $\eta$. 

Notice, that despite the convexity of the cost function Problem~\eqref{eq:JCC-OPF} is non-convex as its feasibility set is non-convex for a sufficiently high level of uncertainty. 

\subsection{Scenario Approach}

Over the last two decades, the scenario approach \cite{nemirovski2006scenario,calafiore2006scenario} 
remains the state-of-the-art method for solving joint chance-constrained optimization. The scenario approach consists in substituting the probabilistic constraints with a larger number of deterministic ones with each constraint standing for some uncertainty realization:
\begin{align}\label{eq:sc-opf}
  \min_x & \; \frac{1}{N} \sum_{t=1}^N \cost(x,\xi_t)\\
  \st & \; p_t^{\min} \le x+\xi^t \le p_t^{\max}, \; 1\le t \le N\nonumber\\
  & \; |\theta_i(\xi^t) - \theta_j(\xi^t)| \le {\bar \theta}_{ij}, (i, j)\in \cE, \; 1\le t \le N\nonumber\\
  & x+\xi^t = B \theta(\xi^t), \; 1\le t \le N\nonumber
\end{align}
where $N$ is a number of scenarios, and $\{\xi^t\}_{i=1}^N$ is a series of uncertainty realization. We assume below that the generation cost, $\cost(x, \xi)$, does not depend on the randomness in power fluctuation, and thus $\cost(x, \xi) = \cost(x)$ for any uncertainty realization $\xi$. Note that dependence of expected cost on known asymptotic parameters of the distribution (variance, mean) are allowed within this framework. We next discuss our method to solve Problem \eqref{eq:sc-opf}.

\section{Algorithm}\label{sec:algo}

\subsection{Idea and Sketch}
The algorithm consists of several steps: constructing an inner approximation to the feasibility set, generating samples outside of the approximation, and, finally, solving the scenario approximation problem \eqref{eq:sc-opf} with a new collection of samples. 

First, using the fact that \emph{the probability of a union of events is bounded from below by the maximum of individual event probabilities}, we construct a lower bound on the probability of constraint feasibility $\mathbb{P}(x+\xi \in \cP)$. The latter allows to add a set of constraints, $x\in \cP_{m}$, so that if $x\not\in \cP_m$ then $\mathbb{P}(x+\xi \in \cP) > \eta$. In other words, if the solution $\bar x$ of the scenario approximation~ \eqref{eq:sc-opf} with samples from the nominal distribution $\cN(0,\Sigma)$ satisfies $\mathbb{P}(x+\xi \in \cP) \ge 1-\eta$, then adding additional inequalities $x\in \cP_{m}$ does not change $\bar x$. 

Second, using the aforementioned bound, we design a polytope $\cP_{in} = \{p: W_{in} \xi \le b_{in}, p = x+\xi\}$ around $x$ with derived $W_{in}$ and $b_{in}$ independent of $x$. Figure~\ref{fig:10} illustrates the idea. Then, we show that for any sample $\xi^t$ and feasible $x$, if $p = x+\xi^t \in \cP_{in}$, then $p$ also necessarily belongs to the constraint feasibility set $\cP$. To this end scenarios $\xi^t: x+\xi^t \in \cP_{in}$ can be eliminated from the optimization problem (see Eq.~\eqref{eq:sc-opf}) without impacting the approximation accuracy. 

\begin{figure}
  \centering
  \begin{tikzpicture}[scale=0.65, node distance={15mm}, thick, main/.style = {draw, scale=.9}] 
  \coordinate (x) at (0, 0);
  \draw (x) circle (1pt);
  \draw (x) ++(0.1, 0.1) node[above right] (tmp) {$x$};
	\coordinate (a) at ( 4.755282581475767 , 1.545084971874737 );
	\coordinate (b) at ( 3.061616997868383e-16 , 5.0 );
	\coordinate (c) at ( -4.755282581475767 , 1.5450849718747375 );
	\coordinate (d) at ( -2.9389262614623664 , -4.045084971874736 );
	\coordinate (e) at ( 2.9389262614623646 , -4.045084971874738 );
  \draw (a) -- (b);
  \draw (b) -- (c);
  \draw (c) -- (d);
  \draw (d) -- (e);
  \draw (e) -- (a);
  \draw
    (a) ++(0.1, 0.01) node[above] (tmp) {$\mathcal{P}$};
    
	\coordinate (ag) at ( 3.804226065180614 , 1.2360679774997896 );
	\coordinate (bg) at ( 2.4492935982947064e-16 , 4.0 );
	\coordinate (cg) at ( -3.804226065180614 , 1.23606797749979 );
	\coordinate (dg) at ( -2.351141009169893 , -3.2360679774997894 );
	\coordinate (eg) at ( 2.3511410091698917 , -3.2360679774997902 );
  \draw [olive] (ag) -- (bg);
  \draw [olive] (bg) -- (cg);
  \draw [olive] (cg) -- (dg);
  \draw [olive] (dg) -- (eg);
  \draw [olive] (eg) -- (ag);
  \draw
    (bg) ++(0.1, 0.01) node[below, olive] (tmp) {$\mathcal{P}_{m}$};
  
	\coordinate (ar) at ( 0.9510565162951535 , 0.3090169943749474 );
	\coordinate (br) at ( 6.123233995736766e-17 , 1.0 );
	\coordinate (cr) at ( -0.9510565162951535 , 0.3090169943749475 );
	\coordinate (dr) at ( -0.5877852522924732 , -0.8090169943749473 );
	\coordinate (er) at ( 0.5877852522924729 , -0.8090169943749476 );
  \draw [red] (ar) -- (br);
  \draw [red] (br) -- (cr);
  \draw [red] (cr) -- (dr);
  \draw [red] (dr) -- (er);
  \draw [red] (er) -- (ar);
  \draw
    (br) ++(1em, .1em) node[above right, red] (tmp) {$\mathcal{P}_{in}$};
  
  \coordinate (b_) at (1, 5.7265);
  \draw
    (b_) node[above right] (tmp) {$\Delta_i$};
  \coordinate (bg_) at (1.5, 5.0898);
  \draw [dashed] (b) -- (b_);
  \draw [dashed] (bg) -- (bg_);
  \draw [thick, black, -latex'] (b_) -- (bg_);
  \draw [thick, black, -latex'] (bg_) -- (b_);
  
  \coordinate (br_) at (-1.5, -0.0898137920080413);
  \draw
    (br_) node[above left] (tmp) {$\Delta_i$};
  \coordinate (x0_) at (-1.0, -0.7265425280053608);
  \draw [dashed] (br) -- (br_);
  \draw [dashed] (x) -- (x0_);
  \draw [thick, black, -latex'] (br_) -- (x0_);
  \draw [thick, black, -latex'] (x0_) -- (br_);
  \end{tikzpicture} 
  \caption{We consider operating point $x$ and derive a polytope $\cP_{m}$ so that no optimal solution of Prob.~ \eqref{eq:JCC-OPF} belongs to $\cP\setminus \cP_{m}$. As it is shown on the figure, $\cP_{m}$ induces a polytope $\cP_{in}$ so that none of the power generations $p\in\cP_{in}$ affect the scenario approximation solution. }
  \label{fig:10}
\end{figure}
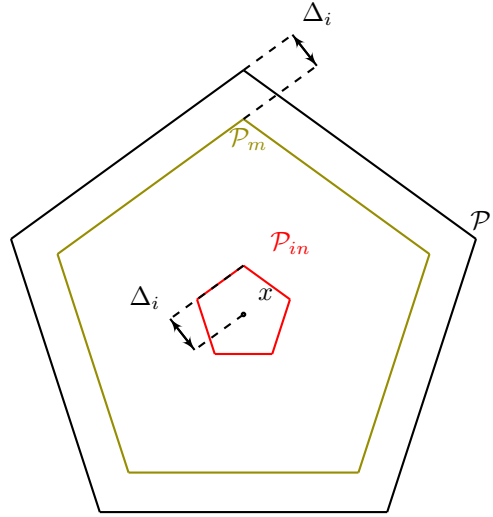

Finally, we sample scenarios outside of the polytope $\cP_{in}$ using the state-of-the-art importance sampling methods \cite{owen2019importance,lukashevich2021power} and solve the Scenario Approximation Problem~\eqref{eq:sc-opf} with the collection of samples generated from importance distribution. Later in this section, we provide rigorous proof of the algorithm efficiency and justify its empirical performance in Section~\ref{sec:emp}. 

\subsection{Inner Approximation}

Consider a probability for the power generation $p$ of being inside the feasibility polytope, $p\in \cP$:
\begin{align*}
  \mathbb{P}&(p\in \cP) = 
  \mathbb{P}(p: Wp \le b) = \\
  & \mathbb{P}\left(p: \bigcap_{i=1}^J \omega_i^\top p \le b_i\right) = 1 - \mathbb{P}\left(p: \bigcup_{i=1}^J \omega_i^\top p > b_i\right)\le\\
  & 1 - \max_{1\le i\le J}\mathbb{P}\left(p: \omega_i^\top p > b_i\right).
\end{align*}

So, if there exists $x$ such that for some $i$, $\mathbb{P}(\omega_i^\top p > b_i) > \eta$, then $\mathbb{P}(p\in \cP) < 1-\eta$. Thus, to satisfy the joint chance constraint for $p = x+\xi$, $\xi\sim \cN(0, \Sigma)$, we need
\begin{align}
  \eta \;\ge\; & \mathbb{P}(\omega_i^\top p > b_i) = \mathbb{P}(\omega_i^\top \xi + \omega_i^\top x > b_i) = \nonumber\\
  & 
  \mathbb{P}\left(\frac{\omega_i^\top \xi}{\|\Sigma^{1/2}\omega\|_2} > \frac{b_i - \omega_i^\top x}{\|\Sigma^{1/2}\omega\|_2}\right) = \nonumber\\
  & \mathbb{P}\left(\zeta > \frac{b_i - \omega_i^\top x}{\|\Sigma^{1/2}\omega\|_2}\right) = \Phi\left(\frac{\omega_i^\top x - b_i}{\|\Sigma^{1/2}\omega\|_2}\right), \label{eq:cnv}
\end{align}
where $\zeta\sim \cN(0,1)$, $\Phi$ is a CDF of the standard normal distribution. Notice, that the function $\Phi(\cdot)$ is convex as soon as its argument is negative, i.e., Ineq.~\eqref{eq:cnv} is convex as soon as~$x\in \cP$. 

A set of inequalities $\eta \ge \Phi\left((\omega_i^\top x - b_i)/(\|\Sigma^{1/2}\omega_i\|_2)\right)$, $1\le i \le J$, defines a polytope $\cP_m$ as follows: 
\begin{align}\label{eq:05}
  \cP_m = \left\{x: \omega_i^\top x \le b_i - \Delta_i\right\},
\end{align}
where $\Delta_i = \|\Sigma^{1/2}\omega_i\|_2 \Phi^{-1}(1-\eta)$, $\eta \le 1/2$. 
Figure~\ref{fig:10} illustrates mutual arrangement of the polytopes $\cP_m \subset\cP$. Note that $\cP_m$ defines an outer approximation of the non-convex chance-constrained feasibility set, which is itself an inner approximation of the constraint feasiblity set without any uncertainty. Eq.~\eqref{eq:05} implies 
Theorem~\ref{thm:10}.
\begin{theorem}\label{thm:10}
  The joint chance-constrained optimal power flow problem~ \eqref{eq:JCC-OPF} has the same set of optimal solutions $X$ as 
  \begin{align}\label{eq:JCC-OPF-Ad}
  \min_x & \;\mathbb{E}_{\xi\sim \cN(0, \Sigma)} \cost(x,\xi)\\
   \st\; & \mathbb{P}_{\xi\sim \cN(0, \Sigma)} (x+\xi \in \cP) \ge 1 - \eta, 0 < \eta \le 1/2,\nonumber\\
   & \; x\in \cP_m\nonumber 
\end{align}
\end{theorem}
\begin{proof}
Equations $x\in\cP_m$ does not affect the solution of Problem~$\eqref{eq:JCC-OPF}$ as the feasiblity set of the chance-constrained optimization problem exists inside $\cP_m$. 
%
\end{proof}



\subsection{Redundant Scenarios}\label{sec:obs}

Another useful consequence of the fact that the optimal solution of the chance-constrained optimal power flow problem is well separated from the boundary of the polytope $\cP$ is that certain scenarios may be removed as they do not improve the accuracy of scenario approximation \eqref{eq:sc-opf}.

Let the optimal solution $\bar x$ of the problem~\eqref{eq:sc-opf} be feasible for the chance-constrained OPF problem. Then by Theorem~\ref{thm:10}, it necessarily also belongs to $\cP_m$. Mathematically, 

\begin{theorem}\label{thm:20}
Any solution of the Scenario approximation problem~ \eqref{eq:sc-opf} that is feasible for the Chance-constrained OPF problem~ \eqref{eq:JCC-OPF} is also a solution of 
\begin{subequations}
\label{eq:Fin}
  \begin{equation}
  \min_x \; \cost(x)\nonumber
  \end{equation}
  \begin{equation}
  \hspace{-20mm}\st\;\; p_t^{\min} \le x+\xi^t \le p_t^{\max}, \; 1\le t \le N\label{eq:Fin-a}
  \end{equation}
  \begin{equation}
   \hspace{11mm} |\theta_i(\xi^t) - \theta_j(\xi^t)| \le {\bar \theta}_{ij}, (i, j)\in \cE, \; 1\le t \le N\label{eq:Fin-b}
  \end{equation}
  \begin{equation}
  \hspace{-19mm} x+\xi^t = B \theta(\xi^t), \; 1\le t \le N\label{eq:Fin-c}
  \end{equation}
  \begin{equation}
  \hspace{-50mm} x\in \cP_m \label{eq:Fin-d}
  \end{equation}
\end{subequations} 
\end{theorem}
Let $ \cP_{in} = \{\xi: \omega_i^\top \xi \le \Delta_i\}$ with $\Delta_i = \|\Sigma^{1/2}\omega_i\| \Phi^{-1}(1-\eta)$. Theorem~\ref{thm:20} implies that for any $x\in \cP_m$, if $\xi^t\in \cP_{in}$, then automatically $p = x+\xi^t \in \cP$. In other words, one can exclude scenario $\xi \in \cP_{in}$ from Problem~ \eqref{eq:sc-opf} as soon as $p = x+\xi \in \cP$. Figure~\ref{fig:10} illustrates the idea and the geometry of $\cP, \cP_m$, and $\cP_{in}$. 

Theorem~\ref{thm:40} follows from the result of Calafiore and Campi~\cite{calafiore2006scenario} and establishes approximation properties of a solution of the Problem~\eqref{eq:Fin}. Assumption~\ref{asmp:10} is the main technical assumption used in the proof of Theorem~\ref{thm:40}.

\begin{assumption}\label{asmp:10}
Assume that for all possible uncertainty realizations $\xi^1, \dots, \xi^N$, the optimization problem \eqref{eq:Fin} is either infeasible or, if feasible, it attains a unique optimal solution.
\end{assumption}

\begin{theorem}\label{thm:40}
Let $\bar x_N$ be a unique solution of the Scenario optimization Problem~\eqref{eq:Fin} with $N$ i.i.d. samples, so that none of the samples belong to $\cP_{in}$. Moreover, assume that for any $N$ the assumption \ref{asmp:10} is fulfilled. Then for any $\delta \in (0,1)$ and any~$\eta \in (0, 1/2]$, $\bar x_N$ is also a solution for the Chance-constrained optimal power flow Problem~\eqref{eq:JCC-OPF} with probability at least $1-\delta$ if 
\begin{align*}
  N \ge \left\lceil 2\frac{(1-\pi)\ln \frac{1}{\delta}}{\eta} + 2d + 2d (1-\pi) \frac{\ln\frac{2(1-\pi)}{\eta}}{\eta} \right\rceil, 
\end{align*} 
where $d$ is a dimension of the space of controllable generators, and $\pi$ is a probability of a random scenario $\xi$ to belong to $\cP_{in}$, and $\pi < 1$. 
\end{theorem}
\begin{proof}
First, notice that discarding scenarios $\xi^t\sim \cN(0, \Sigma)$ is equivalent in solving the scenario approximation problem with $\xi \sim D \Leftrightarrow \xi\sim \cN(0, \Sigma) \st \xi\not\in \cP_{in}$. 

According to the result of Calafiori and Campi~\cite{calafiore2006scenario}, for any probability $\delta\in (0,1)$ and any confidence threshold probability $\varepsilon$, and dimension of the space of parameters $d$ one has, for $N_1$
\begin{align}\label{eq:mon}
  N_1 \ge \left\lceil \frac{2}{\varepsilon} \ln \frac{1}{\delta} + 2d + \frac{2d}{\varepsilon} \ln \frac{2}{\varepsilon}\right\rceil 
\end{align}
scenarios from $D$ and the optimal solution 
$\bar x$ of the Problem~\eqref{eq:Fin}, 
the probability of failure is bounded as
\begin{align*}
  \mathbb{P}_D(p\not\in \cP) \le \varepsilon
\end{align*}
with probability at least $1-\delta$. 

Notice, that the bounds on the number of samples (see Eq.~\eqref{eq:mon}) is strictly decreasing in $\varepsilon$ for $\varepsilon \in (0,1)$. As scenarios in $\cP_{in}$ do not cause failure, to get a probability of failure $\eta$ according to measure $\cN(0, \Sigma)$, we need the failure probability according to $D$ to be at least $\varepsilon = \eta/(1-\pi)$.  
Thus, taking $\varepsilon = \eta/(1-\pi)$ and using monotonicity of~Ineq.~\eqref{eq:mon} one gets the statement of the theorem. 
\end{proof}
Theorem~\ref{thm:40} establishes the number of scenarios sufficient for the scenario approximation solution being feasible for the chance-constrained optimal power flow problem. This number significantly decreases if one can come up with a sufficiently accurate inner approximation of the feasibility set. Notice, that without an inner approximation, i.e. for $\pi = 0$, one gets the result of \cite[Theorem 1]{calafiore2006scenario}. 

\subsection{Importance Sampling}

Although scenario optimization with scenarios that do not belong to the polytope $\cP_{in}$ obey a nice complexity bound, it requires on average $1/(1-\pi)$ or more samples to generate at least one point outside of $\cP_{in}$ and decreases the overall efficiency of the approach. The problem is especially challenging when dealing with rare events, i.e., the confidence level $\eta \to 0$. 

Importance sampling is a general technique that helps to improve the efficiency of scenario generation~\cite{tokdar2010importance}. It consists of a change of the probability distribution to sample rare events with a higher probability. Figure~\ref{fig:20} illustrates the concept. 


\pgfmathdeclarefunction{gauss}{3}{%
 \pgfmathparse{#3/(#2*sqrt(2*pi))*exp(-((x-#1)^2)/(2*#2^2))}%
}
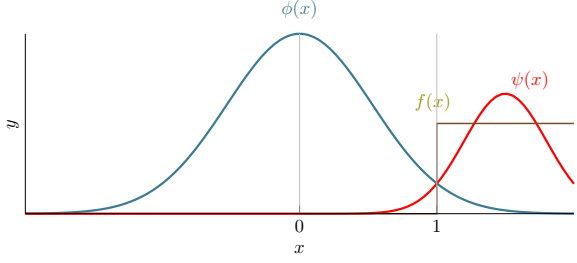
\begin{figure}[ht]
  \centering
\begin{tikzpicture}[scale=0.7]
\begin{axis}[
 no markers, domain=-2:2, samples=100,
 axis lines*=left, xlabel=$x$, ylabel=$y$,
 height=5cm, width=12cm,
 xtick={0,1}, ytick=\empty,
 enlargelimits=false, clip=false, axis on top,
 grid = major
 ]
 \addplot+[very thick,cyan!50!black] {gauss(0,0.53,0.53)} node[above=0.15cm,pos=.5,cyan!50!black] {$\phi(x)$};
 \addplot+[very thick,red] {gauss(1.5,0.3,0.2)} node[above right=0.65cm,pos=0.8,red] {$\psi(x)$};
 \addplot+[const plot, no marks, thick] coordinates {(0,0) (1,0.2) (1,0.2) (2,0.2)} node[above left=0.1cm,pos=0.6,olive] {$f(x)$};
\end{axis}
\end{tikzpicture}
\caption{
  Our goal is to sample points in the area $\{x: f(x) = 1\}$. Sampling from the nominal distribution $\phi(x)$ is inefficient as one needs to sample from the distribution tail. The probability distribution $\psi(x)$ is a better choice as the probability of getting $f(x) = 1$ is substantially higher when sampling from it.}
  \vspace{-5mm}
  \label{fig:20}
\end{figure}

Unfortunately, there is no exact and time-efficient algorithm for sampling outside of a convex polytope from Gaussian distribution \cite{khachiyan1993complexity}; however, the ALOE algorithm~\cite{owen2019importance,lukashevich2021power} proposes an elegant way for approximating the distribution of interest by sampling from a mixture of distributions. 

We consider Gaussian fluctuations of power injections, $\xi \sim \cN(0, \Sigma)$, with known covariance $\Sigma \in \mathbb{R}^{n\times n}$ and aim to sample scenarios outside of $\cP_{in}$ so that the probability distribution to sample from is as close as possible to the conditional Gaussian distribution $\xi \sim \cN(0, \Sigma) \st \xi\not\in \cP_{in}$. 

The method essentially samples from a weighted mixture of conditional Gaussian distributions $D_i$ 
\[\xi \sim D_i \Longleftrightarrow \xi \sim \cN(0, \Sigma) \st \omega_i^\top \xi > \Delta_i. \]

Consider the set of inequalities $\{\omega_i^\top \xi > \Delta_i\}_{i\le J}$ in more detail. First, let $\zeta \sim \cN(0, I_n)$ then the system is equivalent to $\{(\Sigma^{1/2}\omega_i)^\top 
\zeta > \Delta_i\}_{i\le J}$.

Distribution $D_i$ can be simulated exactly using the inverse transform method \cite{l2009monte,morlet1983sampling} that admits conditional sampling $\xi \sim \cN(0, \Sigma)$ s.t. $\omega_i^\top \xi \ge \Delta_i$. 
\begin{enumerate}
  \item Sample $z \sim \cN(0, I)$ and sample $u \sim U(0,1)$
  \item Compute $y = \Phi^{-1}(\Phi(-\Delta_i) + u(1 - \Phi(-\Delta_i)))$
  \item Set $\phi = \phi y + (I - \phi\phi^\top) z$, $\phi = \Sigma^{1/2} \omega_i / \|\Sigma^{1/2} \omega_i\|_2$
  \item Set $\xi = \Sigma^{1/2} \phi$.
\end{enumerate}

In \cite{owen2019importance,lukashevich2021power}, the authors proposed a slightly refined version of the algorithm above that exhibits better numerical stability. We refer to the same papers for the corresponding proofs and analysis. Figure~\ref{fig:conv_vs_MC} illustrates the idea.

Finally, ALOE proposes to sample scenarios from a weighted mixture
\begin{align}
\label{eq:q_d}
  D = \sum_{i=1}^J \alpha_i D_i, \; & \alpha_i \ge 0, \sum_{i=1}^J \alpha_i = 1 \nonumber 
  \\& \alpha_i \propto \Phi(-\Delta_i/\|\Sigma^{1/2}\omega_i\|_2),
\end{align}
where $\Phi$ is a CDF of the standard normal distribution. Let $q_D(\xi)$ be a PDF of distribution $D$, then Theorem~\ref{thm:50} established a maximal ratio of the conditional Gaussian density~$\xi\sim\cN(0, \Sigma) \st \xi\not\in\cP_{in}$ and~$q_D(\xi)$. 
\begin{theorem}\label{thm:50}
Let $p(\xi)$, $q_D(\xi)$ be PDFs of $\xi\sim\cN(0, \Sigma)$ $\st \xi\not\in\cP_{in}$, and a mixture density (see Eq.~\eqref{eq:q_d}) resp. Then for any $\xi\not\in\cP_{in}$, we have
\begin{align}\label{eq:M}
  p(\xi) \le M q_D(\xi), M = \frac{\sum_{i\le J} \Phi(-\Delta_i/\|\Sigma^{1/2}\omega_i\|_2)
  }{\max_{i\le J}\Phi(-\Delta_i/\|\Sigma^{1/2}\omega_i\|_2)}, 
\end{align}
where $D$ and $\alpha_i$ are given in Eq.~\eqref{eq:q_d}.
\end{theorem}
\begin{proof}
Let $\phi(\xi)$ be PDF of $\xi\sim\cN(0, \Sigma)$. Notice, that the conditional densities $D_i$ have probability density functions 
\[q_{D_i}(\xi) = \begin{cases}
\phi(\xi)/\Phi(-\Delta_i/\|\Sigma^{1/2}\omega_i\|_2), & \omega_i^\top \xi > \Delta_i\\
0, & \text{ otherwise}
\end{cases}
\]
Thus the density of distribution $D$ is $\sum_{i\le J} \alpha_iq_{D_i}(\xi)$. Similarly, density $p(\xi)$ outside of the polytope $\cP_{in}$ is $\phi(\xi)/\mathbb{P}_{\xi\sim \cN(0, \Sigma)}(\xi\not\in \cP_{in})$ which is less or equal then
$\phi(\xi)/\max_{i\le J}\Phi(-\Delta_i/\|\Sigma^{1/2}\omega_i\|_2)$ for any $\xi \not\in \cP_{in}$. 

Finally, taking the ratio of $q_D(\xi)$ and $p(\xi)$ and using the value of $\alpha_i$s, we get the lemma statement.
\end{proof}

Theorem~\ref{thm:50} implies that if a probability of a set with respect to measure $q_D$ is less or equal then $\varepsilon$, then the probability of the same set with respect to measure $p$ does not exceed $M\varepsilon$. 

\begin{figure}[!t]
  \centering
  \includegraphics[width=.47\textwidth]{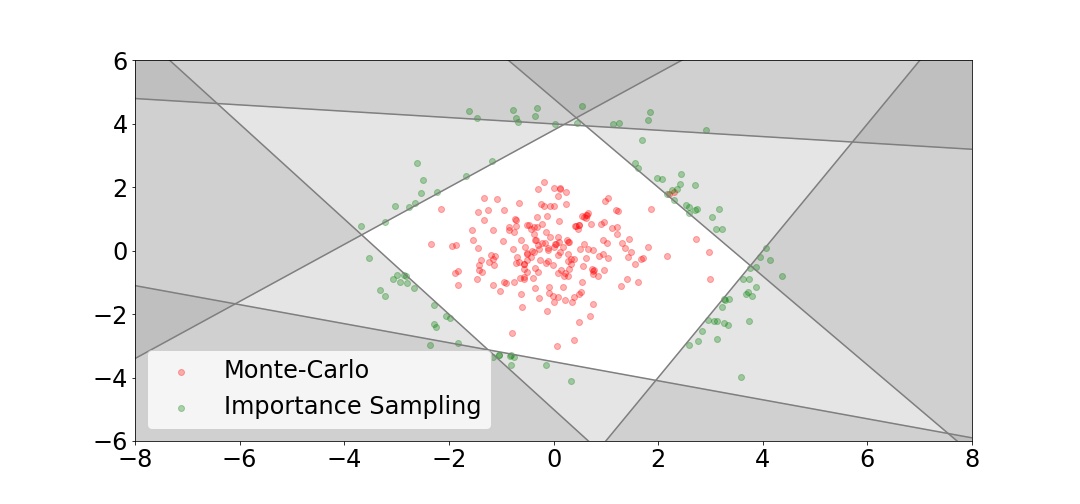}
  \caption{The white area stands for generations that do not exceed operating limits. A set of generations leading to at least one constraint violation is in grey. Two or more reliability constraints are not satisfied in the dark grey area, also samples from a nominal distribution and the constructed mixture marked in red and green respectively}
  \label{fig:conv_vs_MC}
  \vspace{-5mm}
\end{figure}

\subsection{Scenario Approximation with Importance Sampling}

In this section, we present a scenario approximation for the chance-constrained optimal power flow with a set of scenarios generated by the ALOE algorithm~\cite{owen2019importance}. A particular advantage of this approach is that every scenario is generated outside of $\cP_{in}$. The latter substantially improves the accuracy and efficiency of the scenario approximation. In particular, we solve the following optimization problem instead of Problem~\eqref{eq:sc-opf}: 
\begin{subequations} 
\label{eq:FinA}
  \begin{equation}
  \min_x \; \cost(x)\nonumber
  \end{equation}
  \begin{equation}
  \hspace{-20mm}\st\;\; p_t^{\min} \le x+\xi^t \le p_t^{\max}, \; 1\le t \le N\label{eq:FinA-a}
  \end{equation}
  \begin{equation}
   \hspace{11mm} |\theta_i(\xi^t) - \theta_j(\xi^t)| \le {\bar \theta}_{ij}, (i, j)\in \cE, \; 1\le t \le N\label{eq:FinA-b}
  \end{equation}
  \begin{equation}
  \hspace{-19mm} x+\xi^t = B \theta(\xi^t), \; 1\le t \le N\label{eq:FinA-c}
  \end{equation}
  \begin{equation}
  \hspace{-50mm} x\in \cP_m \label{eq:FinA-d}
  \end{equation}
  \begin{equation}
  \hspace{-32mm} \xi^1, \xi^2, \dots, \xi^N\sim D \label{eq:FinA-e},
  \end{equation}
\end{subequations} 
where $D$ is the probability distribution defined by Eq.~\eqref{eq:q_d}. 

Notice, that sampling from distribution $D$ allows to efficiently generate scenarios outside of the polytope $\cP_{in}$. 

However, they follow distribution $D$ instead of $\xi\sim \cN(0, \Sigma) \st \xi\not\in \cP_{in}$. As these distributions are sufficiently close to each other, Theorem~\ref{thm:80} establishes efficient complexity bounds for the scenario approximation with importance sampling. 

\begin{theorem}\label{thm:80}
Let $\bar x_N$ be a unique solution of the Scenario optimization Problem~\eqref{eq:FinA} with $N$ i.i.d. samples follow distribution $D$. Moreover, assume that for any $N$ the assumption \ref{asmp:10} is fulfilled. Then for any $\delta \in (0,1)$ and any~$\eta \in (0, 1/2]$, $\bar x_N$ is also a solution for the Chance-constrained optimal power flow Problem~\eqref{eq:JCC-OPF} with probability at least $1-\delta$ if 
\begin{align*}
  N \ge \left\lceil 2M\frac{(1-\pi)\ln \frac{1}{\delta}}{\eta} + 2d + 2d M(1-\pi) \frac{\ln\frac{2M(1-\pi)}{\eta}}{\eta} \right\rceil, 
\end{align*} 
where $d$ is a dimension of the problem and $\pi$ is a probability of a random scenario $\xi$ to belong to $\cP_{in}$, $\pi < 1$, and constant $M$ is defined by Theorem~\ref{thm:50}. 
\end{theorem}
\begin{proof}
The proof is similar to the one of Theorem~\ref{thm:40}. Application of theorem \ref{thm:50} allows to upper-bound the probability of an event in measure $D$ with respect to its probability in measure $\cN(0,\Sigma) \st \xi\not\in\cP_{in}$. 
\end{proof}

\section{Empirical Study}\label{sec:emp}
We compare the performance of the scenario approximation approaches with and without importance sampling over real and simulated test cases with dimensions varying from several dozens to hundreds of variables. We limit the empirical setting to considering Gaussian distributions and linear feasibility constraints only.

\subsection{Compared Algorithms} 
In this study, we have compared the classical scenario approximation~\cite{calafiore2006scenario} with the importance sampling-based scenario approximation, where the samples are generated with ALOE algorithm~\cite{owen2019importance} from the outside of the polytope $\cP_{in}$ that contains obsolete samples only (see Section~\ref{sec:obs} for details). We omit detailed comparison with other importance sampling strategies \cite{genz2020package,lukashevich2021power,bugallo2017adaptive} when generating scenarios because of the paper space limitation and for the sake of empirical study clarity. 

\subsection{Implementation details} We have used Julia 1.5.3. and PowerModels.jl~\cite{coffrin2018powermodels} on MacBook Pro (2.4GHz, 8-Core Intel Core i9, 64 GB RAM). In the experiments, the computational time for each case has not exceeded 10 minutes, which makes the solution applicable for the operational practice. We will make our code publicly available for research purposes. 

When solving scenario approximation optimization problems, we use CVX~\cite{diamond2016cvxpy} and GLPK optimization solver~\cite{GLPK}. 

\subsection{Test Cases and Numerical Results}

In our experiments we use both synthetic and realistic power flow IEEE test cases \cite{thurner2018pandapower}. 

\paragraph{Synthetic Example}
We first studied efficiency of our algorithm on one dimensional test case:
\begin{align*}
  \max\; & x\\
  \st & \mathbb{P}_\xi(x+\xi \le a) \ge 1-\eta, \xi\sim \cN(0,1)
\end{align*}
for $0 < \eta < 1/2$ and a positive parameter $a$. In this case, the chance-constrained optimization problem admits an exact solution, $x^* = a - \Phi^{-1}(1-\eta)$, the polytopes $\cP_{m}$ and $\cP_{in}$ are $\{x: x\le x^*\}$ and $\xi: \xi \leq \Phi^{-1}(1-\eta)$ respectively. To illustrate the role of an inner approximations $\cP_{m}$ and $\cP_{in}$ we consider different polytopes $\cP_{m} = \{x: x\le b\}$. The latter affects the efficiency of sampling. Figure~\ref{fig:80} illustrates the evolution of scenario approximation without importance with the number of samples and the size of the polytope $\cP_m$.


Although, one can consider an arbitrary polytope $\bar\cP\subset\cP_{in}$ which then will satisfy the conditions of Theorem~\ref{thm:80}. Figure~\ref{fig:80} illustrates the efficiency of the importance sampling approximation for more and less conservative inner polyhedral approximations. 
Notice that a less conservative approximation leads to a less accurate solution we get. To this end, deriving a non-conservative inner approximation is crucial for the importance sampling approach success. 


\begin{figure}[!t]
  \centering
  \vspace{-2mm}
  \includegraphics[width=.45\textwidth]{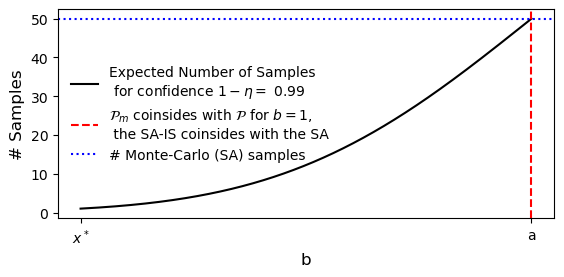}
  \caption{
  The figure demonstrates feasibility of the scenario approximation with importance sampling depending on the size of $\cP_{m}= \{x: x\le b\}$, $b\le a - \Phi^{-1}(1-\eta)$. The more accurate approximation $\cP_m$ (and $\cP_{in}$ resp.) is the less number of samples required by the SA-IS algorithm.}
  \label{fig:80}
  \vspace{-6mm}
\end{figure}

\paragraph{Power grid test cases}

We address the chance-con-
strained DC optimal power flow problem under Gaussian fluctuations of renewables by comparing the algorithms' efficiency in various test cases with up to several hundred buses. We have used DC power flow cases accessible in Matpower~\cite{zimmerman2010matpower}
and PowerModels.jl~\cite{coffrin2018powermodels}.

Table~\ref{tab:emp-ev} summarized the empirical efficiency of the classical scenario approximation and the importance of sampling-based scenario approximation algorithms. For all consider cases (IEEE 30, IEEE 57, and IEEE 118), we assumed the power generation and consumption level to fluctuate with the standard deviation 0.07 of its nominal value. We refer to SA and SA-IS as scenario approximations with and without importance sampling, respectively. 

Our experiments show that with a moderate number of samples the SA approach comes up with a lower cost solution; however, it does not meet security guarantees in out-of sample testing. So the SA method requires a much higher number of samples to deliver constraint satisfaction with the required level of confidence. In contrast, the SA-IS requires much less samples to meet the security constraints with a required probability (validated in out of sample guarantees), with the improvement being significant for higher confidence levels $1-\eta$. Furthermore this incures a very minimal increase in cost, making our approach useful for solving problems with rare events and stricter guarantees. 


\begin{table}[ht]
  \centering
  \begin{tabular}{lllllll}
    \toprule
    Case & Conf.  & DC-OPF & SA & SA-IS & SA & SA-IS\\
       & $1-\eta$ & cost & cost & cost & conf. & conf.\\
    \midrule
    IEEE 30 & 0.95 & 5669 & 5712 & 5735 & 0.86 & \textbf{0.96}\\
    IEEE 30 & 0.99 & 5669 & 5712 & 5760 & 0.86 & \textbf{0.99} \\ 
    IEEE 30 & 0.995 & 5669 & 5712 & 5780 & 0.86 & \textbf{1.00}\\
    \hline
    IEEE 57 & 0.95 & 25016 & 25044 & 25095 & 0.84 & \textbf{0.97}\\
    IEEE 57 & 0.99 & 25016 & 25044 & 25110 & 0.84 & \textbf{1.00}\\ 
    IEEE 57 & 0.995 & 25016 & 25044 & 25165 & 0.84 & \textbf{1.00}\\
    \hline
    IEEE 118 & 0.95 & 84840 & 85607 & 86018 & 0.47 & \textbf{0.96}\\
    IEEE 118 & 0.99 & 84840 & 85607 & 86230 & 0.47 & \textbf{0.99}\\
    IEEE 118 & 0.995 & 84840 & 85607 & 86321 & 0.47 & \textbf{1.00}\\
    \bottomrule
  \end{tabular}
  \caption{Empirical comparison of the scenario approximation approaches with and without importance sampling. In all cases we used \textbf{600} scenarios and solved the problem 50 times over independent samples. SA and SA-AS conf. stand for the average confidence with out of sample test (1000 samples from the nominal uncertainty distribution). }
  \label{tab:emp-ev}
  \vspace{-4mm}
\end{table}




\section{Conclusion}
\label{sec:conclusion}

In this paper, we investigated the scenario approximation for the chance-constrained optimal power flow. We showed that the importance sampling technique used for scenario generation leads to better accuracy and time complexity in theory and practice. Finally, the approach can be extended to automated real-time control of bulk power systems. 



%
\bibliographystyle{IEEEtran}
\bibliography{biblio.bib}

\end{document}